\documentclass[a4paper,11pt]{amsart}
\usepackage{mathptmx}
\usepackage{amsmath}
\usepackage{amscd}
\usepackage{amssymb, bm}
\usepackage{amsthm}
\usepackage{xspace}
\usepackage[all,tips]{xy}
\usepackage[dvips]{graphicx}
\usepackage{verbatim}
\usepackage{syntonly}
\usepackage{hyperref}
\usepackage{graphics, fullpage,color, epsfig,url}
\usepackage{indentfirst}
\usepackage{esint}
\usepackage{enumitem}
\usepackage[dvipsnames]{xcolor}
\usepackage{soul}
\usepackage{tensor}
\usepackage{amsrefs, eucal}
\usepackage{bbm}
\usepackage{ulem}
\usepackage{tikz, wrapfig, epsfig}

\providecommand{\MR}{\relax\ifhmode\unskip\space\fi MR }

\providecommand{\href}[2]{#2}

\theoremstyle{plain}
\newtheorem{thm}{Theorem}[section]
\newtheorem{lem}[thm]{Lemma}
\newtheorem{prop}[thm]{Proposition}

\theoremstyle{remark}
\newtheorem{rem}[thm]{Remark}

\newcommand{\disp}{\displaystyle}

\DeclareMathOperator{\di}{div}

\DeclareMathOperator{\loc}{loc}

\newcommand{\eps}{\varepsilon}
\newcommand{\vp}{\varphi}

\newcommand{\al}{\alpha}
\newcommand{\be}{\beta}
\newcommand{\ga}{\gamma}
\newcommand{\de}{\delta}

\newcommand{\Ga}{\Gamma}

\newcommand{\la}{\lambda}

\newcommand{\Om}{\Omega}

\newcommand{\nid}{\noindent}

\newcommand{\iny}{\infty}
\newcommand{\del}{ \partial}
\newcommand{\su}{\subset}
\newcommand{\LP}{\Delta}
\newcommand{\gr}{\nabla}

\newcommand{\norm}[1]{\left\| #1\right\|}
\newcommand{\innp}[1]{\left< #1 \right>}
\newcommand{\abs}[1]{\left\vert#1\right\vert}

\newcommand{\set}[1]{\left\{#1\right\}}
\newcommand{\brac}[1]{\left[#1\right]}
\newcommand{\pr}[1]{\left( #1 \right) }

\newcommand{\WT}[1]{\ensuremath{\widetilde{#1}}}

\newcommand{\R}{\ensuremath{\mathbb{R}}}

\newcommand{\C}{\ensuremath{\mathbb{C}}}

\date{}

\begin{document}

\title{On Landis' conjecture in the plane \\ for potentials with growth}
\author[Davey]{Blair Davey}
\address{Department of Mathematical Sciences, Montana State University, Bozeman, MT, 59717}
\dedicatory{Dedicated to Carlos Kenig on the occasion of his 70th birthday}
\email{blairdavey@montana.edu}
\thanks{Davey is supported in part by the Simons Foundation Grant 430198 and the National Science Foundation DMS - 2137743.}

\begin{abstract}
We investigate the quantitative unique continuation properties of real-valued solutions to Schr\"odinger equations in the plane with potentials that exhibit growth at infinity.
More precisely, for equations of the form $\LP u - V u = 0$ in $\R^2$, with $\abs{V(z)} \lesssim \abs{z}^{N}$ for some $N \ge 0$, we prove that real-valued solutions  satisfy exponential decay estimates with a rate that depends explicitly on $N$.
The case $N = 0$ corresponds to the Landis conjecture, which was proved for real-valued solutions in the plane in \cite{LMNN20}.
As such, the results in this article may be interpreted as generalized Landis-type theorems.
Our proof techniques rely heavily on the ideas presented in \cite{LMNN20}. 
\end{abstract}

\maketitle

\section{Introduction}

In the late 1960s, E.~M.~Landis \cite{KL88} conjectured that if $u$ is a bounded solution to 
\begin{equation}
\label{ePDE}
\LP u - V u = 0 \; \text{ in } \, \R^n,
\end{equation}
where $V$ is a bounded function and $u$ satisfies $\abs{u(x)} \lesssim \exp\pr{- c \abs{x}^{1+}}$, then $u \equiv 0$.
This conjecture was later disproved by Meshkov \cite{M92} who constructed non-trivial $\C$-valued functions $u$ and $V$ that solve $\LP u - V u = 0$ in $\R^2$, where $V$ is bounded and $\abs{u(x)} \lesssim \exp\pr{- c \abs{x}^{4/3}}$. 
Meshkov also proved a \textit{qualitative unique continuation} result: 
If $\LP u - V u = 0$ in $\R^n$, where $V$ is bounded and $u$ satisfies a decay estimate of the form $\abs{u\pr{x}} \lesssim \exp\pr{- c \abs{x}^{4/3+}}$, then necessarily $u \equiv 0$.

In their work on Anderson localization \cite{BK05}, Bourgain and Kenig established a quantitative version of Meshkov's result. 
As a first step in their proof, they used three-ball inequalities derived from Carleman estimates to establish \textit{order of vanishing} estimates for local solutions to Schr\"odinger equations.
Then, through a scaling argument, they proved a \textit{quantitative unique continuation} result.
More specifically, they showed that if $u$ and $V$ are bounded, and $u$ is normalized so that $\abs{u(0)} \ge 1$, then for sufficiently large values of $R$,
\begin{equation}
 \inf_{|x_0| = R}\norm{u}_{L^\iny\pr{B(x_0, 1)}} \ge \exp{(-CR^{\be}\log R)},
\label{est}
\end{equation} 
where $\be = \frac 4 3$.
Since $ \frac 4 3 > 1$, the constructions of Meshkov, in combination with the qualitative and quantitative unique continuation theorems just described, indicate that Landis' conjecture cannot be true for complex-valued solutions in $\R^2$.
However, at the time, Landis' conjecture still remained open in the real-valued and higher-dimensional settings.
In \cite[Question 1, 2]{Ken06}, Kenig asked if the exponent could be reduced from $\frac 4 3$ down to $1$ in the real-valued setting; and if the related order of vanishing estimate could be improved to match those of Donnelly-Fefferman from \cite{DF88, DF90}. 

In recent years, there has been a surge of activity surrounding Landis' conjecture in the real-valued planar setting.
The breakthrough article \cite{KSW15} by Kenig, Silvestre and Wang proved a quantitative form of Landis' conjecture under the assumption that the zeroth-order term satisfies $V \ge 0$ a.e.
Subsequent papers established analogous results in the settings with drift terms \cite{KW15}, variable coefficients \cite{DKW17}, and singular lower-order terms \cite{KW15, DW20}.
Then we showed that this theorem still holds when $V_-$ exhibits rapid decay at infinity \cite{DKW19}, and when $V_-$ exhibits slow decay at infinity \cite{Dav20a}.
The work of Logunov, Malinnikova, Nadirashvili, and Nazarov \cite{LMNN20} shows that Landis' conjecture holds in the real-valued planar setting.
Their proof uses the nodal structure of the domain along with a domain reduction technique to eliminate any sign condition on the zeroth-order term.
The techniques and ideas from \cite{LMNN20} will be used extensively in this article.

In \cite{Dav14}, I studied the quantitative unique continuation properties of solutions to more general elliptic equations of the form 
$$\LP u + W \cdot \gr u + V u = \la u \; \text{ in } \; \R^n,$$
where $V$ and $W$ exhibit pointwise decay at infinity, and $\la \in \C$.
It was shown that if $\abs{V\pr{x}} \lesssim \innp{x}^{-N}$ and $\abs{W\pr{x}} \lesssim \innp{x}^{-P}$ for $N, P \ge 0$, then the quantitative estimate \eqref{est} holds with $\be = \max \set{1, \frac{4-2N}{3}, 2 - 2P}$.
These results complement those in \cite{CS97}, where analogous qualitative unique continuation theorems are established in the setting where $W \equiv 0$ and $N \in \R$.
By building on the ideas of Meshkov from \cite{M92}, \cite{Dav14} contains examples which prove that the estimates are sharp in certain settings, with further examples in \cite{Dav15}.
These quantitative estimates were generalized in \cite{LW14}, where they proved analogous estimates for solutions to the corresponding equations with variable-coefficient leading terms.

This paper is concerned with proving quantitative unique continuation results for equations of the form \eqref{ePDE}, where $n = 2$, $u$ and $V$ are real-valued, and $V$ exhibits growth at infinity.
We build off of the techniques in \cite{LMNN20} to establish quantitative versions of the results from \cite{CS97} in the setting where $V$ is real-valued and growing (denoted by $\eps \le 0$ in that article).
We now give the precise statement of the main theorem.

\begin{thm}
\label{LandisGrowth}
For some $N \ge 0$, $a_0 > 0$, let $V : \R^2 \to \R$ satisfy the growth condition 
\begin{equation}
\label{growthV}
\abs{V(z)} \le a_0 \abs{z}^{N}.
\end{equation}
Let $u$ be a real-valued solution to \eqref{ePDE} in $\R^2$ with the properties that 
\begin{equation}
\label{solNorm}
\abs{u(0)} = 1
\end{equation}
and for some $c_0 > 0$,
\begin{equation}
\label{solBound}
\abs{u(z)} \le \exp\pr{c_0 \abs{z}^{1 + \frac N 2}}.
\end{equation}
Then there exists constants $C_0 = C_0\pr{a_0, c_0, N} > 0$ and $R_0 > 0$ so that whenever $\abs{z_0} \ge R_0$, it holds that
\begin{equation}
\label{uLower}
\norm{u}_{L^\iny\pr{B(z_0, 1)}} \ge \exp\pr{- C_0 \abs{z_0}^{1+\frac N 2} \log^{\frac 3 2} \abs{z_0}}.
\end{equation}
\end{thm}

The results of \cite{CS97} establish qualitative versions of \eqref{est} with $\be = \frac{4 + 2N}{3} = \frac 4 3 \pr{1 + \frac N 2}$ under the assumption \eqref{growthV} in the complex-valued setting.
Thus, as in the case of bounded $V$, this theorem shows that stronger bounds hold in the real-valued planar setting.

For some $\be > 0$, $c \ne 0$, let 
$$u(z) = \exp\pr{ c \abs{z}^\be}.$$
A computation shows that $\LP u - V u = 0$, where $V(z) := c \be^2 \pr{c\abs{z}^{\be} + 1}\abs{z}^{\be - 2}$
satisfies $\abs{V(z)} \lesssim \abs{z}^{2\be - 2}$.
By setting $\be = 1 + \frac N 2$, this example shows that the theorem is sharp whenever $N \ge 0$.
Based on this example, it seems reasonable to assume that a version of Theorem \ref{LandisGrowth} also holds for potentials that decay at infinity, i.e. for $N < 0$.
To extend the arguments in this paper to decaying potentials, an iterative argument reminiscent of those in \cite{Dav14}, \cite{LW14}, \cite{DKW19}, or \cite{Dav20a} may be needed.
This approach was attempted in the preparation of this manuscript, but the exponent ``got stuck" above $1$ and a resolution to this issue was unclear at the time.
In other words, modifications to the techniques of this paper do not appear to give such results for decaying potentials.
In subsequent articles, we will study both singular potentials and potentials that exhibit decay at infinity.

To prove Theorem \ref{LandisGrowth}, we establish the following local result.
Note that the $R_0 > 0$ here is the same universal constant as in Theorem \ref{LandisGrowth}.

\begin{thm}
\label{localLandis}
Let $u$ be a real-valued solution to $\LP u - V u = 0$ in $B\pr{0, R} \su \R^2$, where$V$ is real-valued and $\norm{V}_{L^\iny(B(0,R))} \le a^2 R^{2 \de}$ for some $\de \ge 0$, $a \ge 1$, $R > 0$.
If $R \ge R_0$, $S \in \brac{\frac R 4, \frac R 2}$, and there exists $M > 0$ so that
\begin{equation}
\label{unormalization}
\sup_{z \in B\pr{0, R - S}} \abs{u(z)}  \ge e^{-M} \sup_{z \in B(0, R)} \abs{u(z)},
\end{equation}
then there exists universal $C_1 > 0$ so that whenever $r \in \pr{0, \frac R {2^{10}}}$, it holds that
\begin{equation}
\label{ulowerBound}
\sup_{z \in B\pr{0, r}} \abs{u(z)} \ge \pr{\frac r R}^{K(R,M)} \sup_{z \in B(0, R-S)} \abs{u(z)},
\end{equation}
where $K(R,M) = C_1 \max \set{a R^{1+\de} \sqrt{\log R}, M+\frac 1 {\log R}}$.
\end{thm}

The proof of this theorem will be presented below in Section \ref{localProof}.
As in \cite{LMNN20}, we reduce the problem to a question about harmonic functions.
Those details are provided in Section \ref{harmonic}.

Assuming that Theorem \ref{localLandis} holds, we present the proof of Theorem \ref{LandisGrowth}.

\begin{proof}[The proof of Theorem \ref{LandisGrowth}]
Fix $z_0 \in \R^2$ with $\abs{z_0} \ge \frac{R_0}2$.
Set $R = 2 \abs{z_0} \ge R_0$ and $S = \frac R 2$.
Define 
$$u_0(z) = u(z_0 + z) \;\quad \text{and} \;\quad V_0(z) = V(z_0 + z)$$ 
so that 
$$\LP u_0 + V_0 \, u_0 = 0 \;\; \text{ in } \; \; B(0, R).$$
Since $\abs{z_0 + z} \le \frac 3 2 R$ for $z \in B(0, R)$, then 
$$\norm{V_0}_{L^\iny\pr{B(0, R)}} \le a_0 \pr{\frac 3 2 R}^N = a_0 \pr{\frac 3 2}^N R^N$$
and 
$$\sup_{z \in B(0, R)} \abs{u_0(z)} \le \exp\brac{c_0\pr{\frac 3 2}^{1 + \frac N 2} R^{1+\frac N 2}}.$$
As 
$$\sup_{B(0, R - S)}\abs{u_0} = \sup_{B(z_0, \abs{z_0})}\abs{u} \ge \abs{u(0)} = 1,$$
then Theorem \ref{localLandis} is applicable with $\de = \frac N 2$, $a = \max\set{\sqrt {a_0} \pr{\frac 3 2}^{\frac N 2}, 1}$, and $M = c_0 \pr{\frac 3 2}^{1 + \frac N2} R^{1+\frac N2}$.
Since 
$$K(R,M) = C_1 \max \set{a R^{1 + \frac N 2} \sqrt{\log R}, c_0 \pr{\frac 3 2}^{1 + \frac N 2} R^{1+\frac N 2}+\frac 1 {\log R}} \le c_1 R^{1+\frac N 2} \sqrt{\log R},$$
where $c_1 = C_1 \brac{a + \pr{\frac{3}2}^{1 + \frac N 2} c_0}$, then
\begin{align*}
\sup_{B\pr{z_0, r}} \abs{u} 
= \sup_{B\pr{0, r}} \abs{u_0} 
\ge \pr{\frac r R}^{c_1 R^{1+N} \sqrt{\log R}} \sup_{B(0, R-S)} \abs{u_0} 
= \pr{\frac r R}^{c_1 R^{1+N} \sqrt{\log R}} \sup_{B(z_0,\frac R 2)} \abs{u} 
\ge \pr{\frac r R}^{c_1 R^{1+N} \sqrt{\log R}} .
\end{align*}
Setting $r = 1$ then shows that
\begin{align*}
\sup_{B\pr{z_0, 1}} \abs{u} 
\ge \exp\pr{- c_1 R^{1+\frac N2} \log^{\frac 3 2} R}
\ge \exp\pr{- C_0 \abs{z_0}^{1+\frac N2}  \log^{\frac 3 2} \abs{z_0}},
\end{align*}
where $C_0 = c_1 2^{1 + \frac N 2} \pr{\frac{10}9}^{\frac 3 2}$.
\end{proof}

The remainder of the article is organized as follows.
In Section \ref{harmonic}, we present and prove a unique continuation theorem for harmonic functions in punctured domains.
As in \cite{LMNN20}, this result for harmonic functions is essential to the proof of Theorem \ref{localLandis}.
We describe this reduction in Section \ref{localProof}, and explain how it implies the proof of Theorem \ref{localLandis}.
We use $c, C$ to denote constants that may change from line to line, while constants with subscripts are fixed.
Unless stated otherwise, all constants are universal.

\section{Decay properties of harmonic functions in punctured domains}
\label{harmonic}

In this section, we present and prove quantitative unique continuation results (in the form of three-ball inequalities) for harmonic functions in punctured domains.
The next section shows how these results lead to the proof of Theorem \ref{localLandis}.
We begin with an application of the Harnack inequality.

\begin{lem}
\label{discBounds}
Let $\set{D_j}$ be a finite collection of $100$-separated unit disks in the plane.
Assume that $h$ is real-valued and harmonic in $\R^2 \setminus \cup D_j$ and that for each $j$, $h$ doesn't change sign in $5D_j \setminus D_j$.
There exists an absolute constant $C_H \ge 10$ for which
\begin{enumerate}
\item $\disp \max_{\del 3 D_j} \abs{h} \le C_H \min_{\del 3 D_j} \abs{h}$
\item $\disp \max_{\del 3 D_j} \abs{\gr h} \le C_H \min_{\del 3 D_j} \abs{h}$.
\end{enumerate}
\end{lem}

\begin{proof}
An application of the Harnack inequality shows that there exists $C_H > 0$ so that for every $j$
\begin{align*}
\max_{\del 3 D_j} \abs{h}
&\le \sup_{4 D_j \setminus 2 D_j} \abs{h} 
\le C_H \inf_{4 D_j \setminus 2 D_j} \abs{h}
\le C_H \min_{\del 3 D_j} \abs{h}.
\end{align*}
For each $z \in \del 3 D_j$, since $h$ doesn't change signs in $B(z, 2)$, an application of Cauchy's inequality as in \cite[Lemma 1.11]{HL11} shows that
\begin{align*}
\abs{\gr h(z)}
&\le \abs{h(z)}
\end{align*}
and the conclusion follows.
\end{proof}

We now state and prove the main result of this section.
The following is a slight modification of the result \cite[Theorem 5.3]{LMNN20}.

\begin{prop}
\label{harmonicProp}
Let $\set{D_j}$ be a finite collection of $100$-separated unit disks in the plane for which $\disp 0 \notin \cup 3 D_j$.
For some $R \ge 2^{10}$, let $h$ be a harmonic function in $B(0, R) \setminus \cup D_j$ with the property that for each $j$, $h$ doesn't change sign in $\pr{5D_j \setminus D_j} \cap B(0,R)$.
Assume that for $S \in \brac{\frac R 4, \frac R 2}$ and for some $M > 0$, it holds that
\begin{equation}
\label{normalization}
\sup_{z \in B\pr{0, R - \frac{S}{32}} \setminus \cup 3 D_j} \abs{h(z)}  \ge e^{-M} \sup_{z \in B(0, R) \setminus \cup 3 D_j} \abs{h(z)}.
\end{equation}
Then for every $r \in \pr{0, \frac R {2^{10}}}$, we have
\begin{equation}
\label{lowerBound}
\sup_{z \in B\pr{0, r} \setminus \cup 3 D_j} \abs{h(z)} 
\ge  \pr{\frac{16 r} R}^{K\pr{R,M}} \sup_{z \in B\pr{0, R - \frac{S}{32}} \setminus \cup 3 D_j} \abs{h(z)},
\end{equation}
where $K(R,M) = \max \set{6C_HR, C_2 M}$, $C_H \ge 10$ is from Lemma \ref{discBounds}, and $C_2 > 0$ is universal. \\
\end{prop}

\begin{rem}
Since this statement, Proposition \ref{harmonicProp}, appears to be very similar to \cite[Theorem 5.3]{LMNN20}, we point out the main differences:
\begin{enumerate}
\item The domain on the lefthand side of \eqref{normalization} depends on $S$ and is therefore variable.
\item The domain on the righthand side of \eqref{lowerBound} matches that on the lefthand side of \eqref{normalization}, while in \cite[Theorem 5.3]{LMNN20}, the domain on the righthand side of \eqref{lowerBound} matches that on the righthand side of \eqref{normalization}.
\item The power $K(R,M)$ here is given as a maximum of two values instead of a sum as in \cite[Theorem 5.3]{LMNN20}.
\item There are differences in the assumed bounds on $R$ and $r$ and therefore constants are different.
\end{enumerate}
\end{rem}

\begin{proof}
We may assume without loss of generality that 
$$\sup_{z \in B\pr{0, R - \frac{S}{32}} \setminus \cup 3 D_j} \abs{h(z)} = 1.$$
Set $k = \max \set{2C_HR, \frac{C_2}{3} M}$, where $C_2$ will be specified below. 
For the sake of contradiction, assume that 
\begin{equation}
\label{contraBound}
\sup_{z \in B\pr{0, r} \setminus \cup 3 D_j} \abs{h(z)} \le \pr{\frac {16r} R}^{3k}.
\end{equation}

Define the punctured annular region
$$\Om := \set{\frac r 2 < \abs{z} < R -1} \setminus \cup 3 D_j$$
and the function
$$f(z) = \frac{h_x - i h_y}{z^k}.$$
Observe that $f$ is analytic in $\Om$ and $\abs{f(z)} = \abs{\gr h(z)} \abs{z}^{-k}$.
We'll analyze the behavior of $f$ over $\Om$.
We begin with bounding $h$ and $\gr h$ over the innermost and outermost parts of the boundary of $\Om$.

Let $W_1$ be the connected component of $\del \Om$ that intersects the inner circle $\set{\abs{z} = \frac r 2}$.
If $z \in W_1$, then there are three cases to consider:
\begin{itemize}
\item[(a)] $\abs{z} \ne \frac r 2$.
\item[(b)] $\abs{z} = \frac r 2$ and there exists $j$ for which $z \in 4 D_j \setminus 3 D_j$.
\item[(c)] $\abs{z} = \frac r 2$ and $z \cap 4 D_j$ is empty for all $j$.
\end{itemize}

\begin{figure}[h]
\label{W1Figure}
\begin{tikzpicture}
\draw[thick] (-1,0) arc (180:450:1cm);
\draw[thick] (0,1) arc (-24.3:294.3:2cm);
\draw[dashed] (0,1) arc (-24.3:360:2cm);
\draw [fill=black] (0,0) circle (1.5pt);
\draw[->] (0,0) -- (0.707, 0.707);
\draw[color=black] (0.5,0.2) node {$\frac r 2$};
\draw [fill=black] (0.177,1.822) circle (2pt);
\draw[color=black] (-1.8,1.822) node {$3D_j$};
\draw[color=black] (0.45,1.822) node {$z_a$};
\end{tikzpicture}
\qquad 
\begin{tikzpicture}
\draw[thick] (-1,0) arc (180:540:1cm);
\draw (-1,1) arc (0:360:2cm);
\draw[dashed] (-0.333,1) arc (0:360:2.666cm);
\draw[color=black] (-3,1) node {$3D_j$};
\draw[color=black] (-2,3) node {$4D_j$};
\draw [fill=black] (0,0) circle (1.5pt);
\draw[->] (0,0) -- (0.707, 0.707);
\draw[color=black] (0.5,0.2) node {$\frac r 2$};
\draw [fill=black] (-0.707, 0.707) circle (2pt);
\draw[color=black] (-0.6,0.45) node {$z_b$};
\draw [fill=black] (0, -1) circle (2pt);
\draw[color=black] (0,-1.3) node {$z_c$};
\end{tikzpicture}
\caption{Possible images of $W_1$ with cases (a), (b) and (c) illustrated by the points $z_a$, $z_b$, and $z_c$, respectively.}
\label{projPics}
\end{figure}
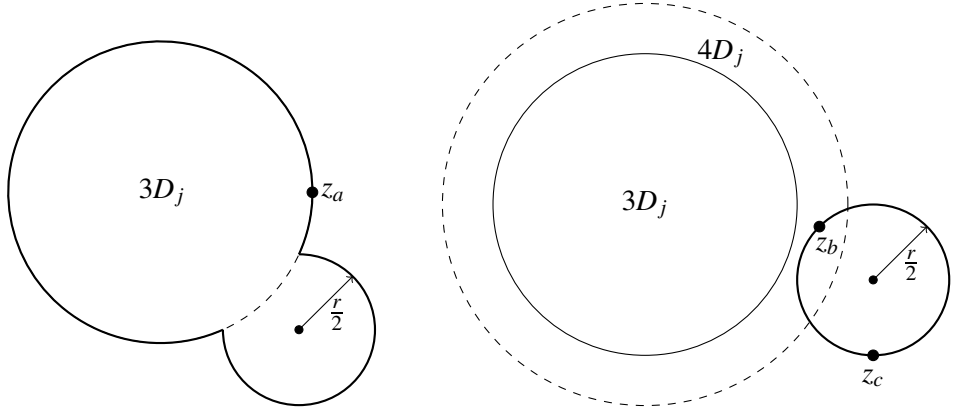

\nid \textbf{Case (a):} There exists $j$ for which $z \in \del 3 D_j$ and $3D_j \cap \set{\abs{z} = \frac r 2}$ is non-empty.
An application of Lemma \ref{discBounds} combined with the fact that $\del 3D_j \cap B(0, r)$ is non-empty shows that
\begin{align*}
\abs{h(z)}, \abs{\gr h(z)} 
&\le C_H \min_{\del 3 D_j} \abs{h} 
\le C_H \sup_{B\pr{0, r} \setminus \cup 3 D_j} \abs{h}
\le C_H \pr{\frac {16r} R}^{3k},
\end{align*}
where the last inequality follows from \eqref{contraBound}.
\\
\textbf{Case (b):} Since $h$ doesn't change signs in $B(z, 1)$, then an application of \cite[Lemma 1.11]{HL11} shows that
\begin{align*}
\abs{\gr h(z)} 
&\le 2 \abs{h(z)}
\le 2 \sup_{B(0, r) \setminus \cup 3 D_j} \abs{h}
\le 2  \pr{\frac {16r} R}^{3k},
\end{align*}
where the second inequality uses that $z \notin \cup 3 D_j$ and we have again applied \eqref{contraBound}.
\\
\textbf{Case (c):} Let $d = \min\set{1, \frac r 2}$ and observe that $B\pr{z, d} \subset B(0, r) \setminus \cup 3 D_j$, so an application of Cauchy's inequality, \cite[Lemma 1.10]{HL11}, shows that
\begin{align*}
\abs{\gr h(z)} 
&\le \frac{2}{d} \sup_{B\pr{z, d}} \abs{h}
\le \frac{2}{d} \sup_{B(0, r) \setminus \cup 3 D_j} \abs{h}
\le \frac 2 d  \pr{\frac {16r} R}^{3k}.
\end{align*}
If $d = 1$, since $k > 1$, then $\disp \frac 2 d \pr{\frac{16 r} R}^k =  2 \pr{\frac{16 r} R}^k < \frac{32r}{R} < \frac 1 {2^5} < \frac 1 2$.
On the other hand, if $d = \frac r 2$, then $\disp \frac 2 d \pr{\frac{16 r} R}^k = \frac 4 r \pr{\frac{16 r} R}^k = \frac{64}{R} \pr{\frac{16 r} R}^{k-1} < \frac 1 {2^4} < \frac 1 2$.

Since $k \ge 2C_H R \ge 2^{11} C_H$, then $2^{10k} \ge \max\set{C_H, 2} = C_H$ and $\pr{\frac r R}^k \le 2^{-10k} \le \frac 1 {\max\set{C_H, 2}}$.
Therefore, by combining all three cases, we see that
\begin{align}
\label{W1Bound}
\sup_{W_1} \abs{h}, \; \sup_{W_1} \abs{\gr h} \le \pr{\frac {16r} R}^{2k}.
\end{align}

Let $W_2$ be the connected component of $\del \Om$ that intersects the outer circle $\set{\abs{z} = R-1}$ and note that $W_2 \su \overline{B(0, R-1)} \setminus B(0, R-7)$.
Now if $z \in W_2$, there are two cases to consider:
\begin{itemize}
\item[(a)] there exists $j$ for which $z \in 4 D_j$.
\item[(b)] $\abs{z} = R-1$ and $z \cap 4 D_j$ is empty for all $j$.
\end{itemize}

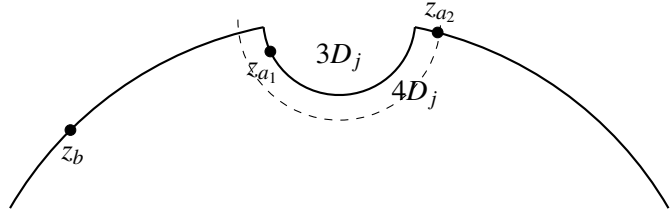
\begin{figure}[h]
\label{W1Figure}
\begin{tikzpicture}
\draw[thick] (4.33,2.5) arc (30:78.52:5cm);
\draw[thick] (-4.33,2.5) arc (150:101.48:5cm);
\draw[thick] (-0.995,4.9) arc (185.739:354.261:1cm);
\draw[dashed] (-1.333,5) arc (180:360:1.333cm);
\draw[color=black] (0,4.5) node {$3D_j$};
\draw[color=black] (1,4) node {$4D_j$};
\draw [fill=black] (1.294, 4.8296) circle (2pt);
\draw[color=black] (1.34, 5.1) node {$z_{a_2}$};
\draw [fill=black] (-0.9063, 4.577) circle (2pt);
\draw[color=black] (-1, 4.3) node {$z_{a_1}$};
\draw [fill=black] (-3.5355, 3.5355) circle (2pt);
\draw[color=black] (-3.5,3.2) node {$z_b$};
\end{tikzpicture}
\caption{A possible image of $W_2$ with case (a) illustrated by the points $z_{a_1}$ and $z_{a_2}$, and case (b) illustrated by $z_b$.}
\label{projPics}
\end{figure}

\nid\textbf{Case (a):} Since $h$ doesn't change sign in $B(z, 1)$, then an application of \cite[Lemma 1.10]{HL11} shows that
\begin{align*}
\abs{\gr h(z)} 
&\le 2 \abs{h(z)}
\le 2 \sup_{B(0, R) \setminus \cup 3 D_j} \abs{h}
\le 2 e^{M},
\end{align*}
where we have applied \eqref{normalization}.
\\
\textbf{Case (b):} Since $B(z, 1) \su B(0,R) \setminus \cup 3 D_j$, then
\begin{align*}
\abs{\gr h(z)} 
&\le 2 \sup_{B\pr{z, 1}} \abs{h}
\le 2 \sup_{B(0, R) \setminus \cup 3 D_j} \abs{h}
\le 2 e^M.
\end{align*}

By combining both cases, we see that
\begin{equation}
\label{W2Bound}
\sup_{W_2} \abs{h} \le e^M, \; \sup_{W_2} \abs{\gr h} \le 2 e^M.
\end{equation}

Now we'll use these estimates on $h$ to understand the behavior of the function $\disp f = \frac{h_x - i h_y}{z^k}$.
Define the set $\Om_1 \su \Om$ as
$$\Om_1 :=  \set{\frac r 2 < \abs{z} < R - \frac{S}{32}} \setminus \cup 3 D_j.$$
Using containment, assumption \eqref{contraBound}, and our rescaling, we see that
\begin{align*}
\sup_{B(0, \frac r 2) \setminus \cup 3D_j} \abs{h}
\le \sup_{B(0, r) \setminus \cup 3D_j} \abs{h}
\le \pr{\frac r R}^{3k}
< 1 
= \sup_{B\pr{0, R - \frac{S}{32}} \setminus \cup 3 D_j} \abs{h}.
\end{align*}
Therefore, there exists $z_0 \in \Om_1$ for which $\abs{h(z_0)} = 1$.
By \eqref{W1Bound}, we have 
\begin{align*}
\sup_{W_1} \abs{h} \le \pr{\frac {16r} R}^{2k} < \frac 1 2,
\end{align*}
so there exists $z_1 \in W_1$ for which $\abs{h(z_1)} = \al < \frac 1 2$.
Let $\Ga$ be a path in $\Om_1$ from $z_0$ to $z_1$ for which $\ell\pr{\Ga} \le 4 R$.
If we assume that $\abs{\gr h(z)}  < \frac 1 {8R}$ for all $z \in \Ga$, then 
\begin{align*}
\frac 1 2 
< \abs{h(z_1) - h(z_0)} 
= \abs{\int_\Ga \gr h(w) \cdot dw}
\le \int_\Ga \abs{\gr h(w)} \abs{dw}
< \frac 1 {8R} \ell\pr{\Ga}
< \frac 1 2,
\end{align*}
which is impossible so it follows that
\begin{align*}
\sup_{\Om} \abs{\gr h}
\ge \sup_{\Om_1} \abs{\gr h}
\ge \frac 1 {8R}.
\end{align*}
Therefore,
\begin{align*}
\sup_{\Om} \abs{f}
\ge \sup_{\Om_1} \abs{f}
\ge \sup_{\Om_1} \abs{\gr h} \pr{R - \frac{S}{32}}^{-k}
\ge \frac 1 {8R} \pr{R - \frac{S}{32}}^{-k}.
\end{align*}
An application of \eqref{W1Bound} shows that
\begin{align*}
\max_{W_1} \abs{f} 
&\le \max_{W_1} \abs{\gr h} \pr{\frac 2 r}^k
\le \pr{\frac {16r} R}^{2k} \pr{\frac 2 r}^k
= \pr{\frac {2^9 r}{R}}^k R^{-k}
< 2^{-k} R^{-k},
\end{align*}
where we have used that $\frac R r > 2^{10}$.
Since $k \ge 2 C_H R \ge 20 R$, then $2^{-k} < \frac 1 {8R}$.
In particular, by combining the previous two inequalities, we deduce that
\begin{align*}
\max_{W_1} \abs{f} 
< \frac 1 {8R} \pr{R - \frac{S}{32}}^{-k}
\le \sup_{\Om} \abs{f}.
\end{align*}
Similarly, an application of \eqref{W2Bound} shows that
\begin{align*}
\max_{W_2} \abs{f} 
&\le \max_{W_2} \abs{\gr h} \pr{R - 7}^{-k}
\le 2 e^M \pr{\frac{R - \frac{S}{32}}{R - 7}}^{k}  \pr{R - \frac{S}{32}}^{-k} .
\end{align*}
Now
\begin{align}
\label{RBig2}
2 e^M \pr{\frac{R - \frac{S}{32}}{R - 7}}^{k} < \frac 1 {8R} 
&\iff k \log \pr{\frac{R - 7 }{R - \frac S{32}}} > M + \log\pr{16R}.
\end{align}
Since $S \in \brac{\frac R 4, \frac R 2}$ and $R \ge 2^{10}$, then 
\begin{align*}
\log \pr{\frac{R - 7 }{R - \frac S{32}}}
= \log \pr{\frac{1 - 7 \cdot R^{-1} }{1 - \frac S{32R}}}
\ge \log \pr{\frac{1 - 7 \cdot R^{-1} }{1 - \frac 1{128}}}
&\ge c_1:= \log \pr{\frac{1 - 7 \cdot 2^{-10} }{1 - 2^{-7}}} 
> 2^{-10}.
\end{align*}
Since $\disp k \ge \frac{C_{2}}{3} M$, then
$$\frac k 2 \log \pr{\frac{R - 7 }{R - \frac S{32}}} \ge \frac{C_{2} c_1}{6} M,$$
while $k \ge 2 C_H R \ge 20 R$ implies that
$$\frac k 2 \log \pr{\frac{R - 7 }{R - \frac S{32}}} > 2^{-10} 10 R > \log\pr{16 R},$$
since $R \ge 2^{10}$.
If we choose $C_{2} = \frac 6 {c_1}$, then \eqref{RBig2} holds.
Thus, we see that
\begin{align*}
\max_{W_2} \abs{f} 
< \sup_{\Om} \abs{f}.
\end{align*}
Since $f$ is a holomorphic function in $\Om$, then the maximum principle guarantees that $\disp \sup_{\Om} \abs{f} = \sup_{\del \Om} \abs{f}$.
As shown above, the maximum doesn't occur on $W_1$ or $W_2$, so there must exist a disk $3 D_j \su \set{\frac r 2 < \abs{z} < R-1}$ for which $\disp \sup_{\Om} \abs{f} = \sup_{\del 3D_j} \abs{f}$.

Considering only the disks $D_j$ for which $3 D_j \su \set{\frac r 2 < \abs{z} < R-1}$, define $z_j \in \del 3 D_j$ to be the point that is closest to the origin, i.e. has the smallest modulus.
Then set $\disp m_j = \min_{\del 3 D_j} \abs{h}$.
Define $j_0$ to be the index for which
\begin{equation}
\label{j0Defn}
m_{j_0} \abs{z_{j_0}}^{-k} = \max_j m_{j} \abs{z_{j}}^{-k}
\end{equation}
and let $j_1$ be the index for which 
$$\sup_{\Om} \abs{f} = \sup_{\del 3 D_{j_1}} \abs{f}.$$
For any $z \in \Om$, an application of Lemma \ref{discBounds} shows that
\begin{equation}
\label{grhBound}
\abs{\gr h(z)} \abs{z}^{-k} 
\le \abs{z_{j_1}}^{-k} \sup_{\del 3 D_{j_1}} \abs{\gr h(z)} 
\le C_H m_{j_1} \abs{z_{j_1}}^{-k} 
\le C_H m_{j_0} \abs{z_{j_0}}^{-k},
\end{equation}
so we see that 
\begin{equation*}
\sup_{z \in \Om}\abs{\gr h(z)} 
\le C_H m_{j_0} \pr{\frac{\abs{z}}{\abs{z_{j_0}}}}^{k}.
\end{equation*}
Since $\disp \sup_{\Om} \abs{\gr h} \ge \frac 1 {8R}$ and $z \in \Om$ is arbitrary,
$$\frac 1 {8R} \le C_H m_{j_0} \pr{\frac{\abs{z}}{\abs{z_{j_0}}}}^{k}\le C_H m_{j_0} \pr{\frac{2R}{r}}^{k}$$
and then
\begin{equation}
\label{hzj0Bound}
\abs{h(z_{j_0})} \ge m_{j_0} \ge \frac 1 {8C_HR} \pr{\frac{r}{2R}}^{k}.
\end{equation}
Define $s = \inf\set{\tau \le 1 : t z_{j_0} \in \Om \text{ for all } t \in \pr{\tau,1}}$ so that the straight line path defined by $\ga(t) = t z_{j_0}$ for $s < t < 1$ is contained in $\Om$ while $s z_{t_0} \in \del \Om$.
In particular, we may integrate $\gr h$ along $\ga$ to get
\begin{align*}
h(z_{j_0}) - h(s z_{j_0})
&= \int_\ga \gr h\pr{z} \cdot d z
= \int_{s}^1 \gr h(t z_{j_0}) \cdot z_{j_0} dt.
\end{align*}
Applications of \eqref{grhBound} and \eqref{hzj0Bound} show that 
\begin{equation}
\begin{aligned}
\label{integralBound}
\abs{h(s z_{j_0})}
&\ge \abs{h(z_{j_0})} - \abs{\int_{s}^1 \gr h(t z_{j_0}) \cdot z_{j_0} dt}
\ge \abs{h(z_{j_0})} - \abs{z_{j_0} } \abs{\int_{s}^1 C_H m_{j_0}t^{k} dt} \\
&\ge m_{j_0} - \frac{m_{j_0} C_H R }{k+1} 
= m_{j_0} \pr{1 - \frac{C_H R}{k+1}}
> \frac{m_{j_0}}{2},
\end{aligned}
\end{equation}
where the last inequality uses that $k+1 > 2C_HR$.
Since $k \ge 2C_H R$, then $2^{k} > 16 C_H R$ and it follows that 
$$\pr{\frac R r}^k \ge 2^{10k} > 16 C_H R \cdot 2^{9k}.$$
Combining \eqref{integralBound} with \eqref{hzj0Bound} shows that
$$\abs{h(s z_{j_0})} > \frac{m_{j_0}}{2} \ge \frac 1 {16 C_H R} \pr{\frac{r}{2R}}^{k} > 2^{9k} \pr{\frac r R}^k\pr{\frac{r}{2R}}^{k} = \pr{\frac{16 r} R}^{2k}.$$
By comparing this bound with \eqref{W1Bound}, we conclude that $s z_{j_0} \notin W_1$ so it must hold that $s z_{j_0} \in \del 3 D_{j_2}$ for some $3 D_{j_2} \su \set{\frac r 2 < \abs{z} < R-1}$.
Then Lemma \ref{discBounds}, that $\abs{z_{j_2}} \le \abs{s z_{j_0}}$, and \eqref{integralBound} show that 
\begin{align}
\label{j2j0Comp}
m_{j_2} \abs{z_{j_2}}^{-k} 
&\ge \frac{\sup_{\del 3 D_{j_2}} \abs{h}}{C_H}\abs{z_{j_2}}^{-k}
\ge \frac{\abs{h(s z_{j_0})}}{C_H}\abs{s z_{j_0}}^{-k}
\ge \frac{s^{-k}}{2C_H}m_{j_0} \abs{z_{j_0}}^{-k}.
\end{align}
Since $z_{j_0} \in \del 3 D_{j_0}$ and $s z_{j_0} \in \del 3 D_{j_2}$ where $j_0 \ne j_2$, and the balls $\set{D_j}$ are of unit radius and $100$-separated, then $\abs{z_{j_0} - s z_{j_0}} \ge 96$.
After rearrangement, we see that $s^{-k} \ge \pr{1 - \frac{96}{R}}^{-k}$.
Since $10 \le C_H$, $2C_HR \le k$, and $\frac{96}{R} < - \log\pr{1 - \frac{96}{R}}$, then 
\begin{align*}
\log\pr{2C_H}
&< 96 \cdot 2 C_H
\le \frac{96}{R} k
< - k \log\pr{1 - \frac{96}{R}}
\le - k \log s,
\end{align*}
from which it follows that $s^{-k} > 2C_H$.
We then conclude from \eqref{j2j0Comp} that $m_{j_2} \abs{z_{j_2}}^{-k} > m_{j_0} \abs{z_{j_0}}^{-k}$ which contradicts \eqref{j0Defn} and gives the desired contradiction.
In other words, \eqref{contraBound} fails to hold and we see that
\begin{align*}
\sup_{z \in B\pr{0, r} \setminus \cup 3 D_j} \abs{h(z)} 
&> \pr{\frac{16r} R}^{3k}
= \pr{\frac {16r} R}^{3k} \sup_{z \in B\pr{0, R - \frac{S}{32}} \setminus \cup 3 D_j} \abs{h(z)},
\end{align*}
which implies \eqref{lowerBound} by our choice of $k$.
\end{proof}

\section{The proof of Theorem \ref{localLandis}}
\label{localProof}

In this section, we show how Theorem \ref{localLandis} follows from Proposition \ref{harmonicProp}.
This reduction is very similar to that described in \cite{LMNN20} with rescaling changes to account for the size of $V$.

Let $u : B(0,R) \su \R^2 \to \R$ be a solution to 
$$\LP u - V u = 0 \; \text{ in } \, B(0,R),$$ 
where for some $a \ge 1$, $\de \ge 0$,
$$\norm{V}_{L^\iny(B(0,R))} \le a^2 R^{2\de}.$$
Let $F_0$ denote the nodal set of $u$, i.e. 
$$F_0 = \set{z \in \R^2 : u(z) = 0}.$$
Define $z_0 \in \overline{B(0, R-S)}$ to satisfy 
\begin{equation}
\abs{u(z_0)} = \sup_{B(0, R-S)} \abs{u}.
\label{z0Def}
\end{equation}
For $\rho > 0$ to be specified below and $c_s$ a universal constant, there exists a set $F_1 \su B(0, R)$ which consists of a collection of $c_s \rho$-separated closed disks of radius $\rho$ which are also $c_s \rho$-separated from $0$, $z_0$, $F_0$, and $\del B(0,R)$.
Moreover, the set $F_0 \cup F_1 \cup \del B(0,R)$ is a $10 c_s \rho$-net in $B(0, R)$.
A more detailed description of this process is given in \cite[\S 2, Act I]{LMNN20}.

Define $\Om = B(0,R) \setminus \pr{F_0 \cup F_1}$ and $\Om_1 = B(0,R) \setminus F_1$.
As shown in \cite[\S 3.1]{LMNN20}, there exists a constant $c_P$ (depending on $c_s$) so that $\Om$ has Poincare constant bounded above by $c_P \rho^2$.
In particular, since $c_P \rho^2 \norm{V}_{L^\iny(B(0,R))} \le c_P \rho^2 a^2 R^{2\de}$, then by choosing $\rho \ll 1$, we can apply \cite[Lemma 3.2]{LMNN20}.
For $\eps \ll 1$ to be defined later on, let 
\begin{equation}
\label{rhoDef}
\rho = \eps a^{-1} R^{-\de}.
\end{equation}
An application of the arguments in \cite[\S 3.2]{LMNN20} then shows that there exists $\vp : \Om \to \R$ with the properties that
\begin{align}
&\LP \vp - V \vp = 0 \text{ in } \Om
\nonumber \\
&\vp -1 \in W^{1,2}_0(\Om)
\nonumber \\
&\norm{\vp -1}_\iny \le c_b \pr{\rho a R^{\de}}^2 = c_b \eps^2,
\label{vpuBound}
\end{align}
where $c_b$ is a universal constant depending on $c_P$, and we have used \eqref{rhoDef}.
By extending $\vp$ to equal $1$ across $F_0 \cup F_1$, it is then shown in \cite[Lemma 4.1]{LMNN20} that $\disp f := \frac u \vp \in W^{1,2}_{\loc}(B(0,R))$ is a weak solution to the divergence-form equation
$$\di\pr{\vp^2 \gr f} = 0 \; \text{ in } \Om_1.$$

We then introduce the Beltrami coefficient $\mu$, defined as follows:
\begin{equation*}
\mu = \left\{ \begin{array}{ll}
\frac{1 - \vp^2}{1 + \vp^2} \frac{f_x + i f_y}{f_x - i f_y} & \text{ in } \Om_1 \text{ when } \, \gr f \ne 0 \\
0 & \text{ otherwise}
\end{array} \right.
\end{equation*}
Since $\abs{\mu} \lesssim \eps^2$, then as shown in \cite{AIM09}, there exists a $K$-quasiconformal homeomorphism of the complex plane where $K \le 1 + C_K \eps^2$, where $C_K$ depends on $c_b$.
That is, there exists some  $w \in W^{1,2}_{\loc}$ which satisfies the Beltrami equation $\disp \frac{\del w}{\del \overline{z}} = \mu \frac{\del w}{\del z}$. 
In fact, an application of the Riemann uniformization theorem shows that there exists a $K$-quasiconformal homeomorphism $g$ of $B(0, R)$ onto itself with $g(0) = 0$.
Moreover, the function $h : = f \circ g^{-1}$ is harmonic in $g(\Om_1)$.

Mori's Theorem implies that
\begin{equation*}
\frac 1 {16} \abs{\frac{z_1 - z_2}{R}}^K
\le \frac{\abs{g(z_1) - g(z_2)}}{R}
\le \frac 1 {16} \abs{\frac{z_1 - z_2}{R}}^{\frac1 K}.
\end{equation*}
Thus, if we set 
\begin{equation}
\label{epsDef}
\eps = \frac{c_e}{ \sqrt{\log R}}
\end{equation}
for some $c_e > 0$, then $\disp K \in \brac{1, 1 + \frac{C_K c_e^2}{\log R}}$ and $R \simeq R^K \simeq R^{\frac 1 K}$.
By appropriately choosing our (universal) constants $c_s$ and $c_e$, it can be shown that $h$ is harmonic in $B(0, R) \setminus \cup D_j$, where each $D_j$ is a disk of radius $32\rho$.
Moreover, the disks are $3200 \rho$-separated from each other, $0$, and $g(z_0)$, while $h$ doesn't change sign in any of the annuli $100 D_j \setminus D_j$.

Since $g : B(0, R) \to B(0, R)$, then we may rescale the map to get 
$$\tilde g := \frac g {32\rho} : B(0, R) \to B\pr{0, \frac R {32\rho}}$$ 
which is onto with $\tilde g(0) = 0$.
Using \eqref{rhoDef} and \eqref{epsDef}, set
\begin{equation}
\label{tildeRDef}
\WT R 
= \frac R {32 \rho} 
= \frac R {32 \frac{c_e}{ \sqrt{\log R}} a^{-1} R^{-\de}}
= \frac{a}{32 c_e} R^{1 + \de}\sqrt{\log R} =: C_3 a R^{1+\de} \sqrt{\log R},
\end{equation}
where we introduce $C_3 = \frac{1}{32c_e}$.
From here, we see that $\tilde h := f\circ \tilde g^{-1}$ is harmonic in $\tilde g\pr{\Om_1}$.
In particular, $\tilde h$ is harmonic in $B(0, \WT R) \setminus \cup \WT D_j$, where now the $\WT D_j$ are unit disks that are 100-separated from each other, from $0$, and from $\tilde g(z_0)$.
Moreover, $\tilde h$ doesn't change signs on any annuli $5 \WT D_j \setminus \WT D_j$.

For $r \ll 1$, since $g\pr{B(0, r)}$ contains a disk of radius $r_0$, where
$$r_0 \ge \frac R {16} \pr{\frac r R}^K \ge  \frac R {16} \pr{\frac r R}^2,$$
then $\tilde g\pr{B(0, r)} \supset B\pr{0, \tilde r}$, where $\tilde r = \frac{r_0}{32 \rho}$ so that
\begin{equation}
\label{scaleBounds}
\frac{16 \tilde r}{\WT R} \ge \pr{\frac r R}^2.
\end{equation}
Since $\tilde g(0) = 0$, then for $r \ll 1$, it holds that $B\pr{0, \tilde r} \setminus \cup 3 \WT D_j = B\pr{0, \tilde r}$ and then
\begin{align}
\label{smallBallBound}
\sup_{B\pr{0, \tilde r} \setminus \cup 3 \WT D_j} \abs{\tilde h}
&= \sup_{B\pr{0, \tilde r}} \abs{\tilde h}
\le \sup_{\tilde g\pr{B\pr{0, r}}} \abs{f \circ \tilde{g}^{-1}}
= \sup_{B\pr{0, r}} \abs{f}.
\end{align}
Since $u = \vp f$, then the bound on $\vp$ from \eqref{vpuBound} implies that
\begin{equation}
\label{ufComparison}
\pr{1 - c_b \eps^2}  \abs{f(z)}  \le \abs{u(z)} \le \pr{1 + c_b \eps^2}  \abs{f(z)}.
\end{equation}
As $z_0$ is as given by \eqref{z0Def}, then for any $z_1 \in \overline{B\pr{0, R- S}}$, it follows that
\begin{align*}
\abs{u(z_1)}
\le \abs{u(z_0)}
\le \pr{1 + c_b \eps^2} \abs{f(z_0)}.
\end{align*}
Since $z_0 \in \overline{B(0, R - S) \cap \Om}$, then the distortion estimate and the separation of $\tilde g(z_0)$ from $\cup 3 \WT D_j$ implies that $\tilde g(z_0) \in B\pr{0, \WT R - \frac{\WT S}{32}} \setminus \bigcup 3 \WT D_j $, where we introduce
$$\WT S := \frac{S}{32\rho} = C_3 a S R^{\de} \sqrt{\log R}.$$
Combining these observations shows that,
\begin{align}
\label{midBallBound}
\frac 1 {1 + c_b \eps^2} \sup_{B(0, R-S)} \abs{u}
&\le \abs{f(z_0)}
= \abs{\tilde h \circ \tilde g(z_0)}
\le \sup_{B\pr{0, \WT R - \frac{\WT S}{32}} \setminus \cup 3 \WT D_j} \abs{\tilde h}.
\end{align}
Moreover, 
\begin{align}
\label{bigBallBound}
\sup_{B(0, \WT R) \setminus \cup 3 \WT D_j} \abs{\tilde h}
\le \sup_{B(0, \WT R)} \abs{\tilde h}
= \sup_{\tilde g\pr{B(0, R)}} \abs{f \circ \tilde g^{-1}}
= \sup_{B(0, R)} \abs{f}.
\end{align}
Subsequent applications of \eqref{midBallBound}, the assumption \eqref{unormalization} from Theorem \ref{localLandis}, \eqref{ufComparison}, and \eqref{bigBallBound} then show that
\begin{align*}
\sup_{B\pr{0, \WT R - \frac{\WT S}{32}} \setminus \cup 3 \WT D_j} \abs{\tilde h}
&\ge \frac 1{1 + c_b \eps^2} \sup_{B(0, R-S)} \abs{u}
\ge \frac 1{1 + c_b \eps^2} e^{-M} \sup_{B(0, R)} \abs{u} \\
&\ge \frac{{1 - c_b \eps^2}}{1 + c_b \eps^2}  e^{-M} \sup_{B(0, R)} \abs{f}
\ge e^{-\pr{M + \frac{c_d}{\log R}}} \sup_{B(0, \WT R) \setminus \cup 3 \WT D_j} \abs{\tilde h},
\end{align*}
where $c_d$ depends on $c_b$ and $c_e$ as in \eqref{epsDef}.

Set $R_0 = \max \set{2^{10},  \exp\pr{C_3^{-2}}}$.
Since $a \ge 1$, then $R \ge R_0$ implies that $C_3 a R^{\de} \sqrt{\log R} \ge a R^\de \ge 1$ and $R \ge 2^{10}$, so that $\WT R = R C_3 a R^{\de} \sqrt{\log R} \ge 2^{10}$ as well.
As $\disp S \, R^{-1} = \WT S \, \WT R^{-1}$, then the hypotheses of Proposition \ref{harmonicProp} hold with $h$, $\set{D_j}$, $r$, $R$, $S$, and $M$ replaced by $\tilde h$, $\set{\WT D_j}$, $\tilde r$, $\WT R$, $\WT S$, $\WT M := M + \frac{c_d}{\log R}$, respectively. 

Applications of \eqref{ufComparison}, \eqref{smallBallBound}, the conclusion \eqref{lowerBound} from Proposition \ref{harmonicProp}, \eqref{scaleBounds}, and \eqref{midBallBound} show that
\begin{align*}
\frac 1{1 - c_b \eps^2} \sup_{B\pr{0, r}} \abs{u}
&\ge \sup_{B\pr{0, r}} \abs{f}
\ge \sup_{B\pr{0, \tilde r} \setminus \cup 3 \WT D_j} \abs{\tilde h}
\ge \pr{\frac{16 {\tilde r}} {\WT R}}^{K\pr{\WT R, \WT M}} \sup_{B\pr{0, \WT R - \frac{\WT S}{32}} \setminus \cup 3 \WT D_j} \abs{\tilde h} \\
&
\ge \frac 1 {1 + c_b \eps^2} \pr{\frac{r} {R}}^{\WT K\pr{R, M} } \sup_{B(0, R-S)}  \abs{u}, 
\end{align*}
where we have introduced 
\begin{align*}
\WT K\pr{R, M}
&= 2 K\pr{\WT R, \WT M}
= 2 \max \set{6 C_H \WT R, C_2 \WT M}
= \max\set{12 C_H C_3 a R^{1+\de} \sqrt{\log R}, 2C_2 \pr{M + \frac {c_d}{\log R}}}.
\end{align*}
Since $\disp \frac{1 - c_b \eps^2}{1 + c_b \eps^2} \ge e^{- \frac{c_d}{\log R}}$ and $e < 4 \le \pr{2^{10}}^{\frac 1 5} \le \pr{\frac{R}{r}}^{\frac 1 5}$, then with universal $C_1$ and
$$K(R, M) :=  C_1 \max \set{ a R^{1+\de} \sqrt{\log R}, \pr{M+\frac 1 {\log R}}}$$ 
we deduce that
\begin{align*}
\sup_{B\pr{0, r}} \abs{u}
&\ge \pr{\frac{r} {R}}^{K\pr{R,M}} \sup_{B(0, R-S)}  \abs{u}
\end{align*}
and the conclusion described by \eqref{ulowerBound} has been shown. 

\subsection*{Acknowledgements}

The author would like to thank the referees for their careful readings of this manuscript and useful suggestions for improvement.


%

\end{document}